\documentclass[reqno]{amsart}
\usepackage{latexsym}
\usepackage{amssymb,amsthm,amsmath}
\usepackage[dvips]{graphicx}

\theoremstyle{plain}
\newtheorem{theorem}{Theorem}
\newtheorem{lemma}[theorem]{Lemma}

\newtheorem{proposition}[theorem]{Proposition}

\theoremstyle{definition}

\theoremstyle{remark}
\newtheorem{remark}[theorem]{Remark}

\def\d#1{{#1\kern-0.4em\char"16\kern-0.1em}}
\def\D#1{{\raise0.2ex\hbox{-}\kern-0.4em #1}}
\newcounter{zd}

\newcounter{zdr}[subsection]

\def\cal{\mathcal}

\begin{document}
\title[]{ Adjoint operator of Bergman projection and Besov  space $ B_{1}$}

\author{David Kalaj}
\address{Faculty of natural sciences and mathematics, University of Montenegro, 
 D\v zor\v za Va\v singtona b.b. 81000, Podgorica, Montenegro}
\email{davidkalaj@gmail.com}

 \author{ DJORDJIJE VUJADINOVI\'C }
\address{Faculty of natural sciences and mathematics, University of Montenegro, D\v zor\v za Va\v singtona b.b. 81000 Podgorica, Montenegro}
 \email{  djordjijevuj@t-com.me}
 \date{}

\begin{abstract}
The main result of this paper is related to finding two-sided bounds of norm for the adjoint operator $P^{\ast}$ of the Bergman projection $P,$ where $P$ denotes the Bergman projection wich maps $L^{1}(D,d\lambda(z))$ onto the  Besov space $B_{1}.$  Here $d\lambda(z)$ is the M\"obius invariant measure ${(1-|z|^2)^{-2}}{dA(z)}$.  It is shown that $2\leq\| P^{\ast}\|\leq 4.$

\end{abstract}


\maketitle

\section{Introduction}
 The Bergman projections are some of the most important operators acting on the spaces of analytic functions over domains of the complex plane. Finding   necessary and sufficient conditions for which Bergman projection is bounded  is one of main subject of this area of research. \\
 \indent Throughout the  paper  $D$ is the open  unit disc in complex plane ${\bf C}$ and $dA(z)=\frac{1}{\pi}dxdy$ is the normalized Lebesgue area measure. $H(D)$ is as usually the space of all analytic functions on the unit disc.\\
\indent  Now we recall some basic facts from the theory of Bergman spaces. The weighted  Bergman space $A^{p}_{\alpha}, 0< p<\infty$ and $-1<\alpha<+\infty$ is the set of analytic functions on $D$ which belong to the Lebesgue space $L^{p}(D, dA_{\alpha}),$ where $dA_{\alpha}=(\alpha+1)(1-|z|^2)^{\alpha}dA(z).$  In the case when $\alpha =0$ we have ordinary Bergman space denoted by $A^{p}.$ The weighted Bergman space $A^{p}_{\alpha}$ is a closed  Banach subspace of $L^{p}(D, dA_{\alpha}).$  The fact that $A^{2}_{\alpha}$ is a Hilbert subspace of $L^{2}(D,dA_{\alpha})$  gives natural way to define Bergman projection $P_{\alpha}$ as an orthogonal projection from $L^{2}(D,dA_{\alpha})$ onto $A^{2}_{\alpha}$  and it is defined as
 $$P_{\alpha}f(z)=\int_{D}f(w){\cal K_{\alpha}(z,w)}dA_{\alpha}(w),z\in D,$$
 where ${\cal K_{\alpha}}$ is the Bergman reproducing kernel ${\cal K_{\alpha}(z,w)=\frac{1}{(1-z\bar{w})^{2+\alpha}}},z,w\in D.$ Since the previous formula is pointwise and $A^{2}_{\alpha}$ is dense in $A^{1}_{\alpha},$ we have that for any $f\in A^{1}_{\alpha}$
 $$f(z)=\int_{D}\frac{f(w)}{(1-z\bar{w})^{2+\alpha}}dA_{\alpha},\enspace z\in D.$$
\indent   In \cite[Theorem~1.10]{Zhu1} there were given sufficient and necessary conditions for $P_{\beta}$ to be bounded, i.e.
\begin{proposition}
Suppose $-1<\alpha,\beta<+\infty$ and $1\leq p<+\infty.$ Then $P_{\beta }$ is a bounded projection from $L^{p}(D,dA_{\alpha})$ onto $A^{p}_{\alpha}$ if and only if $\alpha+1<(\beta+1)p.$
\end{proposition}
 The  Besov space $B_{p}\enspace\mbox{of}\enspace D,\enspace1<p<+\infty,$ is the space of analytic functions $f$ in $D$ such that
  $$\| f\|_{B_{p}}=\left (\int_{D}(1-|z|^2)^{p}|f'(z)|^{p}d\lambda(z)\right)^{\frac{1}{p}}<+\infty, $$
  where $d\lambda(z)=\frac{dA(z)}{(1-|z|^2)^2}$ is a M\"{o}bius invariant measure on $D.$  By the previous relation we define $\| \cdot \|_{B_{p}}$ which is a complete semi-norm on $B_{p}.$
 The  Besov space $B_{p}$ is a Banach space with the norm $$\| f\|=|f(0)|+\| f\|_{B_{p}}.$$
  When $p=\infty,\enspace B_{\infty}={\cal {B}}$  is  the Bloch space, and it is defined as
  $${\cal B}=\{f\in H(D):\sup_{z\in D}{\{(1-|z|^2)|f'(z)|<+\infty\}}.$$
 The  little Bloch space ${\cal B_{0}}$ is defined as
$${\cal B_{0}}=\{f\in H(D):\lim_{|z|\rightarrow 1^{-}}(1-|z|^2)|f'(z)|=0\}.$$
The Little Bloch space ${\cal B_{0}}$ is a closed subspace of  ${\cal B},$ i.e. more precisely ${\cal B_{0}}$ is the closure  in ${\cal B}$ of the polynomials.\\
\indent The Besov space $B_{1}$ is defined in the different manner. $B_{1}$ is the space of analytic functions $f \enspace\mbox{in}\enspace D,$ which can be presented as $$f(z)=\sum_{n=1}^{\infty}a_{n}\varphi_{\lambda_{n}}(z)\enspace\mbox{for some sequence}\enspace (a_{n})\in l^{1}\enspace\mbox{and}\enspace (\lambda_{n})\in D.$$
  Here $\varphi_{\lambda_{n}}$ denotes M\"{o}bius  transform of $D,\enspace \mbox{i.e.}\enspace \varphi_{\lambda_{n}}(z)=\frac{z-\lambda_{n}}{1-z\overline{\lambda_{n}}}, z\in D.$ $B_{1}$ is a Banach space with the norm

$$\| f\|_{B_{1}}=\inf{\left\{\sum_{n=1}^{\infty}|a_{n}|:f(z)=\sum_{n=1}^{\infty}a_{n}\varphi_{\lambda_{n}(z)}\right\}}.$$
 The representation of the dual spaces of the Besov spaces are important in proving our results. So, let us recall some facts about dualities of Besov spaces \cite[Theorem~5.3.7]{Zhu}.
\begin{proposition}
Under the invariant paring
$$\left<f,g\right>=\int_{D}f'(z)\overline{ g'(z)} dA(z)$$
we have following dualities:\\

\vspace*{5mm}
(1) $B_{p}^{\ast} \cong B_{q}$ if $1\leq p<+\infty\enspace\mbox {and}\enspace \frac{1}{p}+\frac{1}{q}=1;$\\

\vspace*{3mm}
(2) $ {\cal B_{0}}^{\ast}\cong B_{1}.$
\end{proposition}
In this paper by boundedness of Bergman projection on $L^{1}(D,d\lambda)$, we mean that there exists a constant $C>0$ such that
   $$\| Pf\|_{B_{1}}\leq C\| f\|_{L^{1}(D,d\lambda)}.$$
Bergman projection connects the Besov  $B_{p}$ space and $L^{p}(D,d\lambda).$ This relation is expressed in next theorem.
 \begin{proposition}
 Suppose $f\in H(D),1\leq p\leq+\infty.$ Then
 $$f\in B_{p}\Leftrightarrow f\in PL^{p}(D,d\lambda),\enspace\mbox{where}\enspace P\enspace\mbox{is the Bergman projection.}$$
 \end{proposition}
   The proof  of previous theorem is given in \cite{Zhu},  where it is shown that the Bergman projection maps boundedly $L^{p}(D,d\lambda)\enspace\mbox{onto}\enspace B_{p},\enspace 1<p<+\infty.$
In  \cite{Perala} Per\"al\"a has found the exact norm of the following projection $P:L^{\infty}(D,dA)\rightarrow {\cal B}$, i.e. he found that $\| P\|=\frac{8}{\pi}.$ For a generalization to several dimensional case see the work of Kalaj and Markovi\'c \cite{km}.
 The extended Bergman projections from Lebesgue classes onto all Besov spaces on the unit ball in $ C^n$ are defined and investigated in work of H.Turgay Kaptanoglu \cite{Kap}.

In this paper we consider the adjoint Bergman projection for the case $p=1,$ i.e. $P^{\ast}:(B_{1})^{\ast}\rightarrow (L^{1}(D,d\lambda))^{\ast}.$ By Proposition~1 the Bergman projection maps $L^{1}(D,d\lambda)$ onto $B_{1}.$ It is easy to show that Closed graph theorem implies that $P$ is bounded in this case  and therefore $P^{\ast}$ is also bounded. Our main results (Theorem~\ref{main}) implies that its norm $\|P^{\ast}\|$ is bounded by $4$. It remains an open problem to find its exact value.
\section{The result}

Assume that $-2<\alpha\leq -1$ and define $d\lambda_{\alpha}=(1-|z|^2)^{\alpha}dA(z).$ As $d\lambda_{\alpha}\le d\lambda$, it follows that $L^1(D,d\lambda)\subset L^{1}(D,d\lambda_{\alpha})$.
In the sequel we consider the following sub-space of normed space $L^{1}(D,d\lambda_{\alpha})$
$$L_{\alpha}^{1}(D,d\lambda)=\left\{f\in L^{1}(D,d\lambda_{\alpha}):\int_{D}|f(z)|d\lambda(z)<\infty\right\}.$$
 In our first result we prove that $ P:L_{\alpha}^{1}(D,d\lambda)\rightarrow B_{1}$ is unbounded mapping.

We need  the following simple result related to the dual of  $L_{\alpha}^{1}(D,d\lambda).$
\begin{lemma}
The dual space  $(L_{\alpha}^{1}(D,d\lambda))^{\ast}$ is isometrically isomorphic to the space $L^{\infty}(D,d\lambda_{\alpha}).$
\end{lemma}
\begin{proof}
By Hahn-Banach  theorem we can join to every bounded functional $\varphi\in(L_{\alpha}^{1}(D,d\lambda))^{\ast}$ its extension $\psi$ on $L^{1}(D,d\lambda_{\alpha}),$ i.e. $\psi\in (L^{1}(D,d\lambda_{\alpha}))^{\ast}=L^{\infty}(D,d\lambda_{\alpha}),$ where
$$\|\psi\|\leq\| \varphi\|.$$
On the other hand, the fact that $L_{\alpha}^{1}(D,d\lambda)$ is dense in $L^{1}(D,d\lambda_{\alpha})$ makes clear that
$$\|\psi\|=\| \varphi\|,$$ and that $\psi$ is unique.\\
Similarly, to every $\psi\in (L^{1}(D,d\lambda_{\alpha}))^{\ast}$ we can join its  bounded restriction on $L_{\alpha}^{1}(D,d\lambda_{\alpha})$ with the same properties as above.
So we conclude that appropriate correspondence is an isometric isomorphism.
\end{proof}

\begin{theorem}\label{teo5}
 Assume that $d\lambda_{\alpha}=(1-|z|^2)^{\alpha}dA(z),\enspace-2<\alpha\leq-1$ and let $P$ be the Bergman projection $P:L^{1}(D,d\lambda_{\alpha})\rightarrow B_{1}.$ Then $P$ is an unbounded operator.
\end{theorem}
\begin{proof}

   According to the Lemma 4, we are going to identify the dual space $(L_{\alpha}^{1}(D,d\lambda_{\alpha}))^{\ast}$ with $L^{\infty}(D,d\lambda_{\alpha}).$ Now, we are going to consider adjoint operator $P^{\ast}:{\cal B}\rightarrow L^{\infty}(D,d\lambda_{\alpha}).$
Let us first determine the form of $P^{\ast}.$ \\
\indent At the beginning we can consider function $g$ as a polynomial or some other function in Bloch space with bounded first derivative. Using classical definition of conjugate operator we have
\begin{equation}
\label{basic1}
P^{\ast}g(f)=g(Pf),\enspace\mbox{where}\enspace g\in {\cal B},f\in L^{1}(D,d\lambda).
\end{equation}
Identity in (1) is equivalent with

\begin{equation}
\label{basic2}
\begin{split}
\int_{D}f(z)\overline{P^{\ast}g(z)}d\lambda_{\alpha}(z) &=\int_{D}(Pf)'(z)\overline{g'(z)}dA(z)\\&= \int_{D}\int_{D}\frac{2f(w)\bar{w}}{(1-z\bar{w})^3}dA(w)\overline{g'(z)}dA(z).\end{split}
\end{equation}
In order to apply Fubini's theorem on the right hand side of \eqref{basic2}, we assume that $f$ is a continuous  function with a compact support in $D,$ i.e. $f\in C_{c}(D)$. Now, we are in position to use  Fubini's theorem, and we get
\begin{equation}\label{23}
\int_{D}f(w)\overline{P^{\ast}g(w)}d\lambda_{\alpha}(w)=2\int_{D}f(w)(1-|w|^2)^{-\alpha}\bar{w}\int_{D}\frac{\overline{g'(z)}}{(1-z\bar{w})^3}dA(z)d\lambda_{\alpha}(w).
\end{equation}
On the left and right hand side of \eqref{23} we have two bounded functionals on $L^{1}(D,d\lambda),$ which are identical on the set $C_{c}(D),$ so we conclude that they are the same.\\
In other words, we have
\begin{equation}
P^{\ast}g(z)=2(1-|z|^2)^{-\alpha}z\int_{D}\frac{g'(w)}{(1-z\bar{w})^3}dA(w),\enspace z\in D.
\end{equation}
So far we have determined $P^{*}$ on the subspace of bounded functions in ${\cal B}$ and in sequel we shall prove that $P^{*}$ is unbounded on this space w.r.t the norm of ${\cal B}$ .\\
\indent Let us observe functions $$g_{z}^{n}(w)=\frac{1}{C_{n}}\sum_{k=0}^{n}\bar{z}^{k}w^{k+1},\enspace w\in D,\enspace\mbox{ for fixed} \enspace z\in D,\enspace n\in {\mathbf N}$$ and $C_{n}=1+\sum_{k=1}^{n}\left(\frac{k}{k+1}\right)^{\frac{k}{2}}.$ It is easy to show  that $g_{z}^{n}\in {\cal B_{0}}\enspace\mbox{and} \| g_{z}^{n}\|_{{\cal B}}\leq 1. $ Also, we should notice that $C_{n}\asymp n,\ \ n\rightarrow +\infty.$
\emph{The notation $A(n) \asymp  B(n) ,\ \ n\rightarrow +\infty$, means that there are constants
$C > 0$ and $c > 0$ such that $cA(n) = B(n) = CA(n)$, for all $n$ large enough.}
\indent On the other hand we have
\begin{equation}
P^{\ast}g_{z}^{n}(z)=2(1-|z|^2)^{-\alpha}z\int_{D}\frac{(g_{z}^{n}(w))'}{(1-z\bar{w})^3}dA(w),
\end{equation}
for fixed $z\in D.$ By using  Taylor expansion of the function  $\frac{1}{(1-z\bar{w})^3},$ the formula $(g_{z}^{n}(w))'=\frac{1}{C_{n}}\sum_{k=0}^{n}(k+1)\bar{z}^{k}w^{k}$ and the orthogonality, for fixed $z\in D$, we get
\begin{equation}
\begin{split}
P^{\ast}g_{z}^{n}(z)&=\frac{2}{C_{n}}(1-|z|^2)^{-\alpha}z\int_{D}\sum_{k=0}^{n}(k+1)\bar{z}^{k}w^{k}\sum_{l=0}^{\infty}\frac{\Gamma{(3+l)}}{l!\Gamma(3)}z^{l}\bar{w}^{l}dA(w)\\
&=\frac{1}{C_{n}}(1-|z|^2)^{-\alpha}z\sum_{k=0}^{n}(k+1)\frac{\Gamma(k+3)}{k!}|z|^{2k}\int_{D}|w|^{2k}dA(w)\\
 &=\frac{1}{C_{n}}(1-|z|^2)^{-\alpha}z\sum_{k=0}^{n}(k+1)(k+2)|z|^{2k}.
\end{split}
\end{equation}

 Now, we see that $\lim_{n\rightarrow +\infty}\sum_{k=0}^{n}(k+1)(k+2)|z|^{2k}=\frac{2}{(1-|z|^2)^{3}},z\in D,$ i.e. more precisely for $z\in D$
  \begin{equation}\label{7}
  \begin{split}
  \sum_{k=0}^{n}&(k+1)(k+2)|z|^{2k}\\&=\frac{-2+|z|^{2+2 n} \left(6+5 n+n^2-2 (1+n) (3+n) |z|^2+(1+n) (2+n) |z|^4\right)}{\left(|z|^2-1\right)^3}.
  \end{split}
  \end{equation}
  Choosing the sequence $z_{n}=1-\frac{1}{n}$, according to the relation \eqref{7}, we estimate
  $$|P^{\ast}g_{z_{n}}^{n}(z_{n})|\asymp n^{2+\alpha},\enspace n\rightarrow +\infty.$$
\end{proof}

Our next result is related to finding a positive constant $C,$ such that
$$\| P^{\ast}f\|\leq C\| f\|,\enspace f\in{\cal B}.$$

 In order to estimate the adjoint operator of the Bergman projection, when $P:L^{1}(D,d\lambda)\rightarrow B_{1},$ we need the following lemma.
 \begin{lemma}
 Let ${\cal B}$ be the Bloch space. If $f\in {\cal B}$   then
  $$\sup_{z\in D}{(1-|z|^2)^{2}|(z^{2}f'(z))'|}\leq 4\| f\|_{{\cal B}}.$$

 \end{lemma}

 \begin{proof}

By using Cauchy integral formula for an  analytic function $f,$ for fixed $z\in D$ and $|z|<r<1$ we get
\begin{equation}
\begin{split}
(1-|z|^2)^{2}&|(z^{2}f'(z))'|=(1-|z|^2)^{2}\left|\frac{1}{2\pi i}\int_{|\xi|=r}\frac{\xi^{2} f'(\xi)}{(\xi-z)^{2}}d\xi\right|\\
&=(1-|z|^2)^{2}\left|\frac{1}{2\pi i}\int_{0}^{2\pi}\frac{r^{3}e^{3it}if'(re^{it})}{(re^{it}-z)^{2}}dt\right|\\
&\leq \frac{r^{3}\| f\|_{{\cal B}}}{2\pi(1-r^2)}(1-|z|^2)^{2}\int_{0}^{2\pi}\frac{dt}{|re^{it}-z|^{2}}dt\\
&=\frac{r^{3}\| f\|_{{\cal B}}}{2\pi(1-r^2)}(1-|z|^2)^{2}\frac{1}{r^2}\int_{0}^{2\pi}\frac{1}{1-\frac{z}{r}e^{-it}}\frac{1}{1-\frac{\bar{z}}{r}e^{it}}dt\\
&=\frac{r\| f\|_{{\cal B}}}{(1-r^2)}(1-|z|^2)^{2}\sum_{n=0}^{\infty}|z|^{2n}r^{-2n}\\
&=r^{3}\| f\|_{{\cal B}}\frac{(1-|z|^2)^{2}}{(1-r^2)(r^{2}-|z|^2)}\\
&\leq\| f\|_{{\cal B}}\frac{(1-|z|^2)^{2}}{(1-r^2)(r^{2}-|z|^2)}.
\end{split}
\end{equation}

On the other hand,  for a fixed $z$ we can choose $r$  such that function $\varphi(r)=(1-r^2)(r^{2}-|z|^2)$ is maximal, i.e. for $r=\sqrt{\frac{1+|z|^2}{2}}$ we have maximal value of function $\varphi_{max}=\frac{(1-|z|^{2})^2}{4}.$  Finally,
$$\sup_{z\in D}{(1-|z|^2)^{2}|(z^{2}f'(z))')|}\leq 4\| f\|_{{\cal B}}.$$

 \end{proof}
The following theorem is  the main result of this paper.
\begin{theorem}\label{main}
Let $P$ be the Bergman projection $P:L^{1}(D,d\lambda)\rightarrow B_{1}.$ Then $P^{\ast}$ is bounded operator and
$$2\le \| P^{\ast}\|\ \leq 4.$$

\end{theorem}

\begin{proof}
As we have already seen in proof of Theorem~\ref{teo5}, the adjoint operator $P^{\ast}:{\cal B}\rightarrow L^{\infty}$ acts in following manner

\begin{equation}
\label{basic11}
\int_{D}f(z)\overline{P^{\ast}g(z)}d\lambda(z) =\int_{D}(Pf)'(z)\overline{g'(z)}dA(z)= 2\int_{D}\int_{D}\frac{f(w)\bar{w}}{(1-z\bar{w})^3}dA(w)\overline{g'(z)}dA(z),
\end{equation}
where $f\in L^{1}(D,d\lambda)$ and $g$ is Bloch function.\\
By using elemental relation for norm of functional  $x^{\ast}$ on Banach space $X,$ \\
 $\| x^{\ast}\|=\sup{\{|\left<x,x^{\ast}\right>|: x\in X,\enspace \| x\|\leq 1\}},$ we conclude that

\begin{equation}
\label{basic22}
\| P^{\ast}g\| =2\sup_{\| f\|\leq 1} {\left|\int_{D}\int_{D}\frac{f(w)\bar{w}}{(1-z\bar{w})^3}dA(w)\overline{g'(z)}dA(z)\right|}.
\end{equation}

It is clear that in the last relation we can take $f$ to be a continuous function with compact support, i.e. $f\in C_{c}(D).$ So, for fixed $f\in C_{c}(D),\enspace\| f\|_{L^{1}(D,d\lambda)}=1,$ we have

\begin{equation}\label{again}
\begin{split}
&\left|\int_{D}\int_{D}\frac{f(w)\bar{w}}{(1-z\bar{w})^3}dA(w)\overline{g'(z)}dA(z)\right|\\
&=\lim_{r\rightarrow 1^{-}}{\left|\int_{rD}\int_{D}\frac{f(w)\bar{w}}{(1-z\bar{w})^3}dA(w)\overline{g'(z)}dA(z)\right|}\\
&=\lim_{r\rightarrow 1^{-}}{\left|\int_{D}f(w)\bar{w}(1-|w|^2)^{2}\int_{rD}\frac{\overline{g'(z)}}{(1-z\bar{w})^3}dA(z)d\lambda(w)\right|}.
\end{split}
\end{equation}
The duality argument $(L^{1}(D,d\lambda))^{\ast}=L^{\infty}(D,d\lambda)$ implies
\begin{equation}
\begin{split}
&\lim_{r\rightarrow 1^{-}}{\left|\int_{D}f(w)\bar{w}(1-|w|^2)^{2}\int_{rD}\frac{\overline{g'(z)}}{(1-z\bar{w})^3}dA(z)dAd\lambda(w)\right|}\\
&\leq\lim_{r\rightarrow 1^{-}}{\sup_{z\in D}(1-|z|^2)^{2}{\left|z\int_{rD}\frac{g'(w)}{(1-z\bar{w})^3}dA(w)\right|}}\\
&=\lim_{r\rightarrow 1^{-}}{\sup_{z\in D}(1-|z|^2)^{2}{\left|z\int_{D}\frac{g'(rw)}{(1-rz\bar{w})^3}rdA(w)\right|}}\\
&\leq\lim_{r\rightarrow 1^{-}}\sup_{z\in D}(1-|z|^2)^{2}|z||\varphi(rz)|,
\end{split}
\end{equation}
where $\varphi(z)=\int_{D}\frac{g'(rw)}{(1-z\bar{w})^3}dA(w)$ is an analytic function of $z$. Moreover \begin{equation}\label{equ}\lim_{r\rightarrow 1^{-}}\sup_{z\in D}(1-|z|^2)^{2}|z||\varphi(rz)|=\sup_{z\in D}(1-|z|^2)^{2}|z||\varphi(z)|.\end{equation}
Let us prove that \eqref{equ}. For $\omega=re^{it},t\in[0,2\pi),\enspace r\in [0,1),$ by using the fact that $|\varphi(\omega z)|$ is subharmonic for fixed $z,$ we conclude that the function $ \sup_{z\in D}(1-|z|^2)^{2}|\omega z\varphi(\omega z)|$ is also subharmonic in $\omega$. By the maximum principle for the subharmonic functions we obtain
\begin{equation}
\begin{split}
\lim_{r\rightarrow 1^{-}}\sup_{z\in D}(1-|z|^2)^{2}| rz||\varphi(rz)|
&\leq \sup_{\omega\in D}\sup_{z\in D}(1-|z|^2)^{2}|\omega z||\varphi(\omega z)|\\
&\leq\sup_{t\in[0,2\pi)}\sup_{z\in D}(1-|z|^2)^{2}|e^{it} z||\varphi(e^{it}z)|\\
&=\sup_{z\in D}(1-|z|^2)^{2}|z||\varphi(z)|.
\end{split}
\end{equation}
Further if, $z_n$ is a sequence in $D$ such that $$\sup_{z\in D}(1-|z|^2)^{2}|z||\varphi(z)|=\lim_{n\to \infty}(1-|z_n|^2)^{2}|z_n||\varphi(z_n)|,$$ then for $r_n=1-(1-|z_n|)/n$, we have $$\lim_{r\rightarrow 1^{-}}\sup_{z\in D}(1-|z|^2)^{2}|z||\varphi(rz)|\ge (1-|z_n/r_n|^2)^{2}|z_n||\varphi(z_n)|=(1-|y_n|^2)^{2}|r_ny_n||\varphi(r_ny_n)|.$$ Here $|y_n|=|z_n|/(1-(1-|z_n|)/n)<1$ for $n>1$. Since $$\lim_{n\to\infty}\frac{1-|z_n|^2}{1-|z_n/r_n|^2}=1,$$ it follows that $$\lim_{r\rightarrow 1^{-}}\sup_{z\in D}(1-|z|^2)^{2}| rz||\varphi(rz)|\ge \sup_{z\in D}(1-|z|^2)^{2}|z||\varphi(z)|.$$
Finally, we get
\begin{equation}
\| P^{\ast}g\|\leq 2\sup_{z\in D }(1-|z|^2)^{2}{\left|z\int_{D}\frac{g'(w)}{(1-z\bar{w})^3}dA(w)\right|},
\end{equation}
for $g\in {\cal B}.$\\
Our goal is to determine a constant $C$ such that
\begin{equation}
2\sup_{z\in D }(1-|z|^2)^{2}{\left|z\int_{D}\frac{g'(w)}{(1-z\bar{w})^3}dA(w)\right|}\leq C\| g\|_{{\cal B}}.
\end{equation}
If we write $w$ in polar coordinate $w=re^{it}$ we get
\begin{equation}\label{dhor}
\begin{split}
z\int_{D}\frac{g'(w)}{(1-z\bar{w})^3}dA(w)&=z\frac{1}{\pi}\int_{0}^{1}rdr\int_{0}^{2\pi}\frac{g'(re^{it})}{(1-zre^{-it})^3}dt\\
&=z\frac{1}{\pi}\int_{0}^{1}rdr\int_{|\xi|=r}\frac{g'(\xi)\xi^2}{(\xi-zr^2)^3}d\xi.
\end{split}
\end{equation}
Now, Cauchy formula applied to the last integral in \eqref{dhor} leads to
\begin{equation}
\begin{split}
z\frac{1}{\pi}\int_{0}^{1}rdr\int_{|\xi|=r}\frac{g'(\xi)\xi^2}{(\xi-zr^2)^3}d\xi &=z\int_{0}^{1}(g'(\xi)\xi^2)_{|\xi=zr^2}^{''}rdr\\&=
\frac{1}{2}\int_{0}^{z}(g'(\xi)\xi^2)^{''}d\xi\\
&=\frac{1}{2} (z^2g'(z))^{'}.
\end{split}
\end{equation}
By Lemma 5, we get
\begin{equation}
\begin{split}
 2\sup_{z\in D }&(1-|z|^2)^{2}{\left|z\int_{D}\frac{g'(w)}{(1-z\bar{w})^3}dA(w)\right|}\\
 &\leq\sup_{z\in D}{(1-|z|^2)^{2}|(z^{2}g'(z))'|}\\
 &\leq 4\| g\|_{{\cal B}}.
\end{split}
\end{equation}
Finally, we are going to give a lower bond for the $\| P^{\ast}\|$.

Let $g=\frac{1}{2}\log[(1+z)/(1-z)]$. Then $g'=1/(1-z^2)$. We make use of \eqref{again}. We have $$\left<P^* g,f\right>=2\lim_{r\to 1^-}{\left|\int_{D}f(w)\bar{w}(1-|w|^2)^{2}\int_{rD}\frac{\overline{g'(z)}}{(1-z\bar{w})^3}dA(z)d\lambda(w)\right|}.$$
Let us find $$I_r=\int_{rD}\frac{\overline{g'(z)}}{(1-z\bar{w})^3}dA(z).$$ We have

$$I_r=\int_{rD}\frac{1}{1-\bar z^2}\frac{1}{(1-z\bar{w})^3}dA(z).$$

Further $$\frac{1}{1-\bar z^2}\frac{1}{(1-z\bar{w})^3}=\frac{1}{2}\sum_{k=0}^\infty \bar z^{2k}\sum_{j=2}^\infty j(j-1) z^{j-2}\bar{w}^{j-2}.$$ Let $z=\rho e^{it}$. Then $$\frac{1}{1-\bar z^2}\frac{1}{(1-z\bar{w})^3}=\frac{1}{2}\sum_{k=0}^\infty (2k+2)(2k+1)\rho^{4k}\bar w^{2k}+\sum_{l\in\mathbb{Z}\setminus\{0\}} e^{il t} A_l(\rho,w),$$ where the functions $A_l(\rho,w)$, $l\in\mathbb{Z}\setminus\{0\}$  do not depend on $t$. Thus \[\begin{split}I_r&=\frac{1}{2\pi}\sum_{k=0}^\infty \int_0^{2\pi}dt\int_0^r(2k+2)(2k+1)\rho^{4k+1}\bar w^{2k}d\rho\\&=\sum_{k=0}^\infty {(k+1)r^{4k+2}\bar w^{2k}} \\&=\frac{r^2}{(1-r^4\bar w^2)^2}.\end{split}\]
 Hence we have \[\begin{split}\left<P^* g,f\right>&=2\lim_{r\to 1^-}{\left|\int_{D}f(w)\bar{w}(1-|w|^2)^{2}\frac{r^2}{(1-r^4\bar w^2)^2}d\lambda(w)\right|}\\&=2\int_{D}f(w)\frac{\bar{w}(1-|w|^2)^{2}}{(1-\bar w^2)^2}d\lambda(w).\end{split}\]
This implies that $$\|P^* g\|=\sup_{\|f\|\le 1}|\left<P^* g,f\right>|=2\sup_{|w|<1}\left|\frac{\bar{w}(1-|w|^2)^{2}}{(1-\bar w^2)^2}\right|=2.$$ Since the Bloch norm of $g$ is equal to 1, we obtain that $\|P^*\|\ge 2$ as desired.
\end{proof}
\begin{remark}
It remains an open problem to find the exact value of $\|P^{\ast}\|$ and, in view of the proof of  Theorem~\ref{main}, we believe that it is equivalent to the following extremal problem. Given the functional $\mathcal{P}(f):=\sup_{|z|<1}(1-|z|^2)^2|(z^2 f(z))'|$  find its supremum under the condition $f\in\mathcal{B}$, $\|f\|=1$.
\end{remark}

\subsection*{Acknowledgement} We are thankful to the referee for comments and suggestions  that have improved this paper.

\end{document}